\documentclass[oneside,english]{amsart}
\usepackage[T1]{fontenc}
\usepackage[latin9]{inputenc}
\usepackage{amstext}
\usepackage{amsthm}
\usepackage{amssymb}

\makeatletter
\numberwithin{equation}{section}
\numberwithin{figure}{section}
\numberwithin{table}{section}
\theoremstyle{plain}
\newtheorem{thm}{\protect\theoremname}[section]
\theoremstyle{definition}
\newtheorem{defn}[thm]{\protect\definitionname}

\makeatother

\usepackage{babel}
\providecommand{\definitionname}{Definition}
\providecommand{\theoremname}{Theorem}

\begin{document}
\title{Abundance of progression in large set for non commutative semigroup}
\author{Sujan Pal}
\address{Department of Mathematics, University of Kalyani, Kalyani, Nadia-741235,
West Bengal, India}
\email{\emph{sujan2016pal@gmail.com}}
\begin{abstract}
The notion of abundance of certain type of configurations in certain
large sets, first proved by Furstenberg and Glasner in 1998. After
that many authors investigated abundance of different types of configurations
in different types of large sets. Hindman, Hosseini, Strauss and Tootkaboni
recently introduced another notion of large sets called $CR$ sets.
Debnath and De proved abundance of arithmetic progressions in $CR$
sets for commutative semigroups. In the present article we investigate
abundance of progressions in large sets for non-commutative semigroups.
\end{abstract}

\maketitle

\section{Introduction}

A subset $A$ of the set of integers $\mathbb{Z}$ is called syndetic
if gaps in it are bounded, and it is called thick if it contains arbitrary
length intervals. Sets which can be expressed as the intersection
of thick and syndetic sets are called piecewise syndetic. These notions
can be extended for general semigroups. For, general semigroup $(S,\cdot)$,
a set $A\subset S$ is said to be syndetic in if there exists a finite
set $F\subset S$ such that translations by the elements of $F$ cover
$S$. A set $A\subset S$ is said to be thick if for every finite
set $F\subset S$, there exists a translation of $F$ contained in
$A$. Now a set $A\subset S$ is said to be piecewise syndetic set
if can be expressed as the intersection of thick and syndetic sets
are called piecewise syndetic. There are also other equivalent forms
of definitions of piecewise syndetic sets.

One of the famous Ramsey theoretic results is so called van der Waer
den\textquoteright s Theorem \cite{key-7}, which states that at least
one cell of any partition of $\mathbb{N}$ contains arithmetic progression
of arbitrary length. Since arithmetic progressions are invariant under
shifts, it follows that every piecewise syndetic set contains arbitrarily
long arithmetic progressions. The following theorem, which guarantees
the abundance of arithmetic progressions, was first proved by H. Furstenberg
and E. Glasner \cite{key-4}.
\begin{thm}
\label{FW ab}Let $k\in\mathbb{N}$ and assume that $S\subset\mathbb{Z}$
is piecewise syndetic. Then $\left\{ \left(a,d\right):\left\{ a,a+d,...,a+kd\right\} \subset S\right\} $
is piecewise syndetic in $\mathbb{Z}^{2}$.
\end{thm}

In \cite{key-1} author extended theorem \ref{FW ab} for different
types of large sets. In a recent work, \cite{Key-2} Bergelsoon and
Glasscock introduced a new such notion of largeness for commutative
semigroup called $CR$-sets.
\begin{defn}
\label{Comm CR set} Let $\left(S,+\right)$ be a commutative semigroup.
$A\subseteq S$ is said to be a $CR$-set if for each $k\in\mathbb{N}$,
there exists $r\in\mathbb{N}$ such that whenever $M$ is an $r\times k$
matrix with entries from $S$, there exist $a\in S$ and nonempty
$H\subseteq\left\{ 1,2,\cdots,r\right\} $ such that for each $j\in\left\{ 1,2,\cdots,k\right\} $

\[
a+\sum_{t\in H}m_{i,j}\in A.
\]
\end{defn}

Later in \cite{Key-5} Hindman, Hosseini, Strauss and Tootkaboni rephrased
the definition, also introduce finner gradation named $k-CR$ set,
where $k\in\mathbb{N}$. This definition also helps us to understand
the difference between $J$-sets and $CR$-sets.
\begin{defn}
\label{new def cr}Let $\left(S,+\right)$ be a commutative semigroup
and let $A\subseteq S$. 
\begin{enumerate}
\item $A$ is said to be a $J$-set if and only if whenever $F\in\mathcal{P}_{f}\left(^{\mathbb{N}}S\right)$,
there exist $a\in S$ and $H\in\mathcal{P}_{f}\left(\mathbb{N}\right)$
such that for each $f\in F$, 
\[
a+\sum_{t\in H}f\left(t\right)\in A.
\]
\item $A$ is said to be a $CR$ set if for each $k\in\mathbb{N}$ there
exists $r\in\mathbb{N}$ and whenever $F\in\mathcal{P}_{f}\left(\,^{\mathbb{N}}S\right)$
with $\mid F\mid\leq k$ , there exist $a\in S$ and $H\in\mathcal{P}_{f}\left(\mathbb{N}\right)$
with $\max H\leq r$ such that for each $f\in F$, 
\[
a+\sum_{t\in H}f\left(t\right)\in A.
\]
\item $A$ is said to be a $k-CR$ set if there exists $r\in\mathbb{N}$
and whenever $F\in\mathcal{P}_{f}\left(\,^{\mathbb{N}}S\right)$ with
$\mid F\mid\leq k$ , there exist $a\in S$ and $H\in\mathcal{P}_{f}\left(\mathbb{N}\right)$
with $\max H\leq r$ such that for each $f\in F$, 
\[
a+\sum_{t\in H}f\left(t\right)\in A.
\]
\end{enumerate}
\end{defn}

Then clearly $CR$-sets are $J$-sets.

The notion of $CR$-sets have obvious generalization to non commutative
case.
\begin{defn}
Let $\left(S,\cdot\right)$ be a semigroup and let $A\subseteq S$.
\begin{enumerate}
\item $A$ be a\emph{ combinatorially rich set} ($CR$-set) if and only
if for each $k\in\mathbb{N}$ there exist $r\in\mathbb{N}$ and $m\in\mathbb{N}$
such that for each $F\in\mathcal{P}_{f}\left(\,^{\mathbb{N}}S\right)$
with $\mid F\mid\leq k$, there exist $\overrightarrow{a}=\left(a_{1},a_{2},\cdots,a_{m},a_{m+1}\right)\in S^{m+1}$,
and $t_{1}<t_{2}<\cdots<t_{m}\leq r$ in $\mathbb{N}$ such that for
each $f\in F$, 
\[
a_{1}f\left(t_{1}\right)a_{2}f\left(t_{2}\right)\cdots a_{m}f\left(t_{m}\right)a_{m+1}\in A.
\]
\item $A$ be a\emph{ $k-$combinatorially rich set} ($k-CR$-set) if and
only if there exist $r\in\mathbb{N}$ and $m\in\mathbb{N}$ such that
for each $F\in\mathcal{P}_{f}\left(\,^{\mathbb{N}}S\right)$ with
$\mid F\mid\leq k$, there exist $\overrightarrow{a}=\left(a_{1},a_{2},\cdots,a_{m},a_{m+1}\right)\in S^{m+1}$,
and $t_{1}<t_{2}<\cdots<t_{m}\leq r$ in $\mathbb{N}$ such that for
each $f\in F$, 
\[
a_{1}f\left(t_{1}\right)a_{2}f\left(t_{2}\right)\cdots a_{m}f\left(t_{m}\right)a_{m+1}\in A.
\]
\end{enumerate}
\end{defn}

In \cite{Key-5} authors proved partition regularity of $CR$- sets.
\begin{thm}
\label{PRCR}Let $\left(S,\cdot\right)$ be a semigroup and let $A_{1}$
and $A_{2}$ be subsets of $S$. If $A_{1}\cup A_{2}$ is a $CR$-set
in $S$, then either $A_{1}$ or $A_{2}$ is a $CR$-set in $S$. 
\end{thm}

\begin{proof}
\cite[Theorem 2.4]{Key-5}
\end{proof}
In our work we will use the structure of Stone-\v{C}ech compactification
of discrete semigroup. Let $\left(S,\cdot\right)$ be a discrete semigroup
and $\beta S$ be the Stone-\v{C}ech compactification of the discrete
semigroup $S$ and $\cdot$ on $\beta S$ is the extension of $'\cdot'$
on $S$. The points of $\beta S$ are ultrafilters and principal ultrafilters
are identified by the points of $S$. The extension is unique extension
for which $\left(\beta S,\cdot\right)$ is compact, right topological
semigroup with $S$ contained in its topological center. That is,
for all $p\in\beta S$ the function $\rho_{p}:\beta S\to\beta S$
is continuous, where $\rho_{p}(q)=q\cdot p$ and for all $x\in S$,
the function $\lambda_{x}:\beta S\to\beta S$ is continuous, where
$\lambda_{x}(q)=x\cdot q$. For $p,q\in\beta S$, $p\cdot q=\left\{ A\subseteq S:\left\{ x\in S:x^{-1}A\in q\right\} \in p\right\} $,
where $x^{-1}A=\left\{ y\in S:x\cdot y\in A\right\} $.
\begin{defn}
Now we want to introduce a subset of $\beta S$ related to $J$-set
which is two sided ideal.
\begin{enumerate}
\item $J\left(S\right)=\left\{ p\in\beta S:\text{ for all }A\in p,A\text{ is a }J\text{-set }\right\} $
where $\left(S,\cdot\right)$ is a semigroup.
\item $CR\left(S\right)=\left\{ p\in\beta S:\text{ for all }A\in p,A\text{ is a }CR\text{-set }\right\} $
\item For $k\in\mathbb{N}$, $k-CR\left(S\right)=\left\{ p\in\beta S:\text{ for all }A\in p,A\text{ is a }k-CR\text{-set }\right\} $
\end{enumerate}
\end{defn}

\begin{thm}
Let, $\left(S,\cdot\right)$ be a semigroup. Then $CR\left(S\right)$
is a compact two sided ideal of $\beta S$ and for each $k\in\mathbb{N}$,
$k-CR\left(S\right)$ is a compact two sided ideal of $\beta S$ and
hence contain idempotents.
\end{thm}

\begin{proof}
\cite[Thorem 2.6]{Key-5}
\end{proof}

\section{ABUNDANCE OF CR SET }

In \cite{key-3-1}, authors proved the abundance of arithmetic progressions
in $CR$-sets for commutative semigroup using the defintion \ref{Comm CR set}.
Here we giving an alternating proof using the Defintion \ref{new def cr}.
In \cite{Key-5}, authors proved that both the definition are equivalent.
This theorem also agrees with that fact.
\begin{thm}
\label{CA} Let $\left(S,\cdot\right)$ be a commutative semigroup,
$A\subseteq S$ be a $CR$-set in $S$ and $l\in\mathbb{N}$. Then
the collection $\left\{ \left(a,d\right):\left\{ a,a+d,a+2d,\ldots,a+ld\right\} \subseteq A\right\} $
is a $CR$-set in $S\times S$.
\end{thm}

\begin{proof}
Let $C=\left\{ \left(a,d\right):\left\{ a,a+d,a+2d,\ldots,a+ld\right\} \subseteq A\right\} \subset S\times S$.
To show that $C$ is a $CR$-set in $S\times S$, we have to show
there exists $r\in\mathbb{N}$, for every $k-$many functions $f_{1},f_{2},\ldots,f_{k}\in\left(S\times S\right)^{\mathbb{N}}$
there exist $H\subseteq\left\{ 1,2,\ldots,r\right\} $ and $\overline{a}\in S\times S$

\[
\overline{a}+\sum_{t\in H}f_{i}(t)\in C\text{ for all }i\in\left\{ 1,2,\ldots,k\right\} .
\]

Let for each $i\in\left\{ 1,2,\ldots,k\right\} ,$ $g_{i},g_{i}'\in S^{\mathbb{N}}$
be two component functions of $f_{i}$ so that $f_{i}=\left(g_{i},g_{i}'\right)$.
Let us choose $b\in S$ and set 

\[
p_{i,j}=g_{i}+j\left(b+g_{i}'\right)\text{ where }i\in\left\{ 1,2,\ldots,k\right\} \text{ and }j\in\left\{ 1,2,\ldots,l\right\} .
\]

Then $p_{i,j}\in S^{\mathbb{N}}$ for all $i\in\left\{ 1,2,\ldots,k\right\} $
and $j\in\left\{ 1,2,\ldots,l\right\} $. Now consider $lk\in\mathbb{N}$.
Since $A$ is a $CR$-set in $S$, then there exists $r\in\mathbb{N}$,
such that for these $lk-$many functions $\left(p_{i,j}\right)_{i=1,j=1}^{k,l}$,
there exist $a\in S$ and $H\subseteq\left\{ 1,2,\ldots,r\right\} $
such that 

\[
a+\sum_{t\in H}p_{i,j}(t)\in A,\text{ for every }i\in\left\{ 1,2,\ldots,k\right\} ,j\in\left\{ 1,2,\ldots,l\right\} .
\]

Which implies that for each $i$,

\[
a+\sum_{t\in H}\left(g_{i}(t)+j\left(b+g_{i}'\right)(t)\right)\in A,\text{ for all }j\in\left\{ 1,2,\ldots,l\right\} .
\]
Therefore for each $i$,

\[
\left(a+\sum_{t\in H}g_{i}(t)\right)+j\left(b\mid H\mid+\sum_{t\in H}g_{i}'(t)\right)\in A,\text{ for all }j\in\left\{ 1,2,\ldots,l\right\} .
\]

Then we have for each $i$,
\[
\left(a+\sum_{t\in H}g_{i}(t),b\mid H\mid+\sum_{t\in H}g_{i}'(t)\right)\in C,
\]

i.e, for each $i$,
\[
\left(a,b\mid H\mid\right)+\left(\sum_{t\in H}\left(g_{i},g_{i}'\right)(t)\right)\in C
\]

Therefore we get $\overline{a}=\left(a,b\mid H\mid\right)\in S\times S$
and $r$, $H$ as above we have
\[
\overline{a}+\sum_{t\in H}f_{i}(t)\in C\text{ for all }i\in\left\{ 1,2,\ldots,k\right\} .
\]
\end{proof}
In the following we turn our attention to non-commutative semigroups.
First we need to recall the definition of progression in non-commutative
semigroup \cite[Definition 2.1(4)]{key-1}.
\begin{defn}
Let $\left(S,\cdot\right)$ be an arbitrary semigroup. Given $l\in\mathbb{N}$,
a set $B\subseteq S$ is a \emph{length $l$ progression }if there
exist $m\in\mathbb{N}$, $\overline{a}=\left(a_{1},a_{2}\right)\in S^{2}$,
and $d\in S$ such that $B=\left\{ a_{1}d^{t}a_{2}:t\in\left\{ 1,2,\ldots,l\right\} \right\} $.
\end{defn}

It could be preferred generalization of abundance of progression in
a $CR$-set $A$ in $S$ if the collection $\left\{ \left(a_{1},a_{2},d\right):\left\{ a_{1}d^{t}a_{2}:t\in\left\{ 1,2\right\} \right\} \subseteq A\right\} $
is $CR$-set in $S\times S\times S$. But we are unable to solve the
above problem. Putting some extra condition on $CR$-sets we achieve
an analoge result.
\begin{defn}
Let $\left(S,\cdot\right)$ be a semigroup and let $A\subseteq S$.
\begin{enumerate}
\item Then $A$ is a\emph{ strong combinatorially rich set} (a $SCR$-set)
if for each $k\in\mathbb{N}$ there exists $r\in\mathbb{N}$ such
that for each $F\in\mathcal{P}_{f}\left(\,^{\mathbb{N}}S\right)$
with $\mid F\mid\leq k$, there exist $\overline{a}=\left(a_{1},a_{2}\right)\in S^{2}$,
and $t\leq r$ in $\mathbb{N}$ such that for each $f\in F$, 
\[
a_{1}f\left(t\right)a_{2}\in A.
\]
\item Then $A$ is a\emph{ k-strong combinatorially rich set} (a $k-SCR$-set)
if there exists $r\in\mathbb{N}$ such that for each $F\in\mathcal{P}_{f}\left(\,^{\mathbb{N}}S\right)$
with $\mid F\mid\leq k$, there exist $\overline{a}=\left(a_{1},a_{2}\right)\in S^{2}$,
and $t\leq r$ in $\mathbb{N}$ such that for each $f\in F$,
\[
a_{1}f\left(t\right)a_{2}\in A.
\]
\end{enumerate}
\end{defn}

Clearly $SCR$-sets are $CR$-sets. Since $SCR$-set is a particular
$CR$-set $\left(m=1\right)$.
\begin{thm}
\label{PRSCR}Let $\left(S,\cdot\right)$ be a semigroup and let $A_{1}$
and $A_{2}$ be subsets of $S$. If $A_{1}\cup A_{2}$ is a $SCR$-set
in $S$, then either $A_{1}$ or $A_{2}$ is a $SCR$-set in $S$. 
\end{thm}

\begin{proof}
The proof is similar to the proof of theorem \ref{PRCR}.
\end{proof}
Next we define two new subsets of $\beta S$ corresponding to $SCR$-sets.
\begin{defn}
$\left(S,\cdot\right)$ be a semigroup
\begin{enumerate}
\item $SCR\left(S\right)=\left\{ p\in\beta S:\forall A\in p,A\text{ is a }SCR\text{ set}\right\} $
\item For $k\in\mathbb{N}$, $k-SCR\left(S\right)=\left\{ p\in\beta S:\forall A\in p,A\text{ is a }k-SCR\text{ set}\right\} $ 
\end{enumerate}
\end{defn}

Clearly $SCR\left(S\right)=\bigcup_{k\in\mathbb{N}}k-SCR\left(S\right)$.
\begin{thm}
\label{ID}$\left(S,\cdot\right)$ be a semigroup. Then for each $k\in\mathbb{N}$,
$k-SCR\left(S\right)$ is a two sided ideal of $\beta S$.
\end{thm}

\begin{proof}
Let $p\in k-SCR\left(S\right)$ and $q\in\beta S$. We want to show
$pq,qp\in k-SCR\left(S\right)$.

To show, $pq\in k-SCR\left(S\right)$, Let $A\in pq$, so $B=\left\{ x\in S:x^{-1}A\in q\right\} \in p$.
So $B$ is a $k-SCR$ set. Then there exists $r\in\mathbb{N}$ such
that for each $F\in\mathcal{P}_{f}\left(\,^{\mathbb{N}}S\right)$
with $\mid F\mid\leq k$, there exist $\overline{a}=\left(a_{1},a_{2}\right)\in S^{2}$,
and $t_{1}\leq r$ in $\mathbb{N}$ such that for each $f\in F$,
\[
a_{1}f\left(t_{1}\right)a_{2}\in B.
\]

Then 
\[
\bigcap_{f\in F}\left(a_{1}f\left(t_{1}\right)a_{2}\right)^{-1}A\in q.
\]

So we have an element 
\[
y\in\bigcap_{f\in F}\left(a_{1}f\left(t_{1}\right)a_{2}\right)^{-1}A.
\]

So, 
\[
a_{1}f\left(t_{1}\right)a_{2}y\in A,\forall f\in F.
\]

Let us define, $b_{1}=a_{1}$ and $b_{2}=a_{2}y$. Then $\overline{b}=\left(b_{1},b_{2}\right)\in S^{2}$

So, 
\[
b_{1}f\left(t_{1}\right)b_{2}\in A,\forall f\in F.
\]

Therefore $A$ is a $k-SCR$ set. And $A$ is arbitrary element from
$pq$. So, $pq\in k-SCR\left(S\right)$.

Now, If $A\in qp$, Then $B=\left\{ x\in S:x^{-1}A\in p\right\} \in q$.
Then picking $x\in B$, $x^{-1}A\in p$. So, $x^{-1}A$ is a $k-SCR$
set.

Then there exists $r\in\mathbb{N}$ such that for each $F\in\mathcal{P}_{f}\left(\,^{\mathbb{N}}S\right)$
with $\mid F\mid\leq k$, there exist $\overline{a}=\left(a_{1},a_{2}\right)\in S^{2}$,
and $t_{1}\leq r$ in $\mathbb{N}$ such that for each $f\in F$,
\[
a_{1}f\left(t_{1}\right)a_{2}\in x^{-1}A.
\]

So, 
\[
xa_{1}f\left(t_{1}\right)a_{2}\in A,\forall f\in F.
\]

Let us define, $b_{1}=xa_{1}$ and $b_{2}=a_{2}$. Then $\overline{b}=\left(b_{1},b_{2}\right)\in S^{2}$

So, 
\[
b_{1}f\left(t_{1}\right)b_{2}\in A,\forall f\in F.
\]

Therefore $A$ is a $k-SCR$ set. And $A$ is arbitrary element from
$qp$. So, $qp\in k-SCR\left(S\right)$.
\end{proof}
From theorem \ref{PRSCR} and theorem \ref{ID}, we can say $SCR\left(S\right)$
is a compact two sided ideal of $\beta S$.

Now we are going to prove our main result.
\begin{thm}
Let $\left(S,\cdot\right)$ be an arbitrary semigroup and $A\subseteq S$.
Let $A$ be a $SCR$-set in $S$. Then the collection $\left\{ \left(a_{1},a_{2},d\right):\left\{ a_{1}d^{t}a_{2}:t\in\left\{ 1,2\right\} \right\} \subseteq A\right\} $
is a $CR$-set in $S\times S\times S$.
\end{thm}

\begin{proof}
Let, $D=\left\{ \left(a_{1},a_{2},d\right):\left\{ a_{1}d^{t}a_{2}:t\in\left\{ 1,2\right\} \right\} \subseteq A\right\} $.

To show that $D$ is a $SCR$-set in $S\times S\times S$, we have
to show that, For every $k\in\mathbb{N}$ there exists $r\in\mathbb{N}$,
if $f_{i}\in\left(S\times S\times S\right)^{\mathbb{N}}$ for all
$i\in\left\{ 1,2,\ldots,k\right\} $, there exist $\overrightarrow{a}=\left(\overline{a_{1}},\overline{a_{2}}\right)\in\left(S\times S\times S\right)^{2}$
and $t\leq r$ 

\[
\overline{a_{1}}f_{i}\left(t\right)\overline{a_{2}}\in D,\text{ where }i\in\left\{ 1,2,\ldots,k\right\} .
\]

Let, $f_{i}=(h_{i},g_{i},k_{i})$ where for each $i\in\left\{ 1,2,\ldots,k\right\} $
$h_{i},g_{i},k_{i}\in S^{\mathbb{N}}$.

Let us choose $a_{12},a_{13},a_{21},a_{23}\in S$ and set

\[
p_{i}(t)=h_{i}(t)a_{21}a_{13}k_{i}(t)a_{23}a_{13}k_{i}(t)a_{23}a_{12}g_{i}(t)=h_{i}(t)a_{21}\left(a_{13}k_{i}(t)a_{23}\right)^{2}a_{12}g_{i}(t)
\]
 and 
\[
q_{i}(t)=h_{i}(t)a_{21}a_{13}k_{i}(t)a_{23}a_{12}g_{i}(t)\text{ where }i\in\left\{ 1,2,\ldots,k\right\} 
\]
.

Since $A$ is a $SCR$-set in $S$, then for this $2k\in\mathbb{N}$
there exists $r\in\mathbb{N},$ For $F=\left\{ p_{i},q_{i}:i\in\left\{ 1,2,\ldots,k\right\} \right\} \subseteq\mathcal{P}_{f}\left(S^{\mathbb{N}}\right)$
there exist $b=\left(b_{1},b_{2}\right)\in S^{2}$ where $b_{1}=a_{11},b_{2}=a_{22}$
and $t\leq r_{2k}$ in $\mathbb{N}$ such that 

\[
b_{1}f\left(t\right)b_{2}\in A,\text{ for all }f\in F
\]

i.e, 
\[
a_{11}h_{i}(t)a_{21}\left(a_{13}k_{i}(t)a_{23}\right)^{l}a_{12}g_{i}(t)a_{22}\in A,i\in\left\{ 1,2,\ldots,k\right\} ,l\in\left\{ 1,2\right\} .
\]

Then from definition of $D$,

\[
\left(a_{11}h_{i}(t)a_{21},a_{12}g_{i}(t)a_{22},a_{13}k_{i}(t)a_{23}\right)\in D,i\in\left\{ 1,2,\ldots,k\right\} 
\]

i,e, 
\[
\left(a_{11},a_{12},a_{13}\right)\left(h_{i},g_{i},k_{i}\right)(t)\left(a_{21},a_{22},a_{23}\right)\in D,i\in\left\{ 1,2,\ldots,k\right\} 
\]

i,e, 
\[
\overline{a_{1}}f_{i}\left(t\right)\overline{a_{2}}\in D,i\in\left\{ 1,2,\ldots,k\right\} 
\]

Where $\overline{a_{1}}=\left(a_{11},a_{12},a_{13}\right),\overline{a_{2}}=\left(a_{21},a_{22},a_{23}\right)\in S\times S\times S$.

So $D$ is $SCR$-set in $S\times S\times S$. Since $SCR$-sets are
$CR$-sets also, Then $D$ is also $CR$-set in $S\times S\times S$.
\end{proof}
We can extent this theorem for arbitrary length progression easily.
\begin{thm}
Let $\left(S,\cdot\right)$ be an arbitrary semigroup and $A\subseteq S$.
Let $A$ be a $SCR$-set in $S$. Then for each $l\in\mathbb{N}$
the collection 
\[
\left\{ \left(a_{1},a_{2},d\right):\left\{ a_{1}d^{t}a_{2}:t\in\left\{ 1,2,\ldots,l\right\} \right\} \subseteq A\right\} 
\]
 is $CR$-set in $S\times S\times S$.
\end{thm}

\begin{proof}
The proof is similar to the previous. In this case we will take $lk$
functions $p_{ij}(t)=h_{i}(t)a_{21}\left(a_{13}k_{i}(t)a_{23}\right)^{j}a_{12}g_{i}(t)$
where $i\in\left\{ 1,2,\ldots,k\right\} $ and $j\in\left\{ 1,2,\ldots,l\right\} $.
\end{proof}

\end{document}